\documentclass[leqno]{amsart}
\usepackage{amsmath}
\usepackage{amssymb}
\usepackage{amsthm}
\usepackage{enumerate}
\usepackage[mathscr]{eucal}
\usepackage{graphicx}
\theoremstyle{plain}
\newtheorem{theorem}{Theorem}[section]
\newtheorem{prop}[theorem]{Proposition}
\newtheorem{cor}{Corollary}[theorem]

\theoremstyle{definition}
\newtheorem{definition}{Definition}[section]
\newtheorem{remark}{Remark}[section]

\newtheorem{example}{Example}[theorem]

\usepackage[pagewise]{lineno}
\begin{document}

\title[Some remarks on Birkhoff-James orthogonality of linear operators]{Some remarks on Birkhoff-James orthogonality of linear operators}
\author[Debmalya Sain, Kallol Paul and Arpita Mal  ]{ Debmalya Sain, Kallol Paul and Arpita Mal  }

\address[Sain]{Department of Mathematics\\ Indian Institute of Science\\ Bengaluru 560012\\ Karnataka \\India\\ }
\email{saindebmalya@gmail.com}

\address[Paul]{Department of Mathematics\\ Jadavpur University\\ Kolkata 700032\\ West Bengal\\ INDIA}
\email{kalloldada@gmail.com}

\address[Mal]{Department of Mathematics\\ Jadavpur University\\ Kolkata 700032\\ West Bengal\\ INDIA}
\email{arpitamalju@gmail.com}

\thanks{The research of Dr. Debmalya Sain is sponsored by Dr. D. S. Kothari Postdoctoral Fellowship. Dr. Sain feels elated to lovingly acknowledge the immense contribution of his extended family in every sphere of his life, mathematical or otherwise. The research of Prof Kallol Paul  is supported by project MATRICS  of DST, Govt. of India. The research of third author is supported by UGC, Govt. of India.}


\begin{abstract}
We study Birkhoff-James orthogonality of compact (bounded) linear operators between Hilbert spaces and Banach spaces. Applying the notion of semi-inner-products in normed linear spaces and some related geometric ideas, we generalize and improve some of the recent results in this context. In particular, we obtain a characterization of Euclidean spaces and also prove that it is possible to retrieve the norm of a compact (bounded) linear operator (functional) in terms of its Birkhoff-James orthogonality set. We also present some best approximation type results in the space of bounded linear operators.    
\end{abstract}

\maketitle

\section{Introduction and preliminaries}
The purpose of the present article is to generalize and improve some of the recent results on Birkhoff-James orthogonality of bounded (compact) linear operators defined between real normed linear spaces. Let us first establish the notations and the terminologies to be used throughout the article.\\
Let $ \mathbb{X},\mathbb{Y} $ denote normed linear spaces. In this paper, without explicitly mentioning any further, we work with only \emph{real} normed linear spaces. Let $ \mathbb{H} $ denote a Hilbert space. Finite-dimensional Hilbert spaces are also known as Euclidean spaces. Let $ B_{\mathbb{X}} $ and $ S_{\mathbb{X}} $ denote the unit ball and the unit sphere of $ \mathbb{X} $ respectively, i.e., $ B_{\mathbb{X}} = \{ x \in \mathbb{X}~: \| x \| \leq 1 \} $ and $ S_{\mathbb{X}} = \{ x \in \mathbb{X}~: \| x \| = 1 \}. $ We say that $ \mathbb{X} $ is strictly convex if every point of $ S_{\mathbb{X}} $ is an extreme point of the convex set $ B_{\mathbb{X}}. $ Let $ \mathbb{L}(\mathbb{X},\mathbb{Y}) $ and $ \mathbb{K}(\mathbb{X},\mathbb{Y}) $ denote the normed linear space of all bounded and compact linear operators from $ \mathbb{X} $ to $ \mathbb{Y} $ respectively, endowed with the usual operator norm. Given a bounded linear operator $ T \in \mathbb{L}(\mathbb{X},\mathbb{Y}), $ let $ M_T $ denote the collection of unit vectors at which $ T $ attains norm,  i.e., $ M_T = \{x \in S_{\mathbb{X}} : \| Tx \| = \| T \|\}. $ A sequence $\{x_n\}\subseteq S_{\mathbb{X}}$ is said to be a norming sequence for a bounded linear operator $T$ if $\|Tx_n\|\to \|T\|.$ Given a subset $ A $ of $ \mathbb{X}, $ let $ |A| $ denote the cardinality of $ A. $ Birkhoff-James orthogonality \cite{B,J} plays an important role in describing the geometry of normed linear spaces. Given any two elements $ x,y \in \mathbb{X}, $ we say that $ x $ is Birkhoff-James orthogonal to $ y, $ written as $ x \perp_B y, $ if $ \| x+\lambda y \| \geq \| x \| $ for all scalars $ \lambda. $ Naturally, this definition makes sense in the normed linear space of all bounded linear operators. \\

Recently, W\'{o}jcik \cite{W} has studied Birkhoff-James orthogonality of compact linear operators between normed linear spaces and Hilbert spaces and has obtained a characterization of finite-dimensional Hilbert spaces among several applications of the said study. In this article, we illustrate that it is possible to generalize and improve some of the results obtained in \cite{W} by applying a geometric concept introduced in \cite{S}. Given any two elements $ x,y \in \mathbb{X},  $ we say that $ y \in x^{+} $  if $ \| x+\lambda y \| \geq \| x \| $ for all $ \lambda \geq 0. $ Similarly, we say that $ y \in x^{-} $  if $ \| x+\lambda y \| \geq \| x \| $ for all $ \lambda \leq 0. $ Basic geometric properties of a normed linear space related to these notions has been explored in Proposition $ 2.1 $ of \cite{S}. We further require the concept of semi-inner-products (s.i.p.) \cite{G,L} in normed linear spaces, that plays an important role in the whole scheme of things. Let us first mention the relevant definition.

\begin{definition}
Let $ \mathbb{X} $ be a normed space. A function $ [ ~,~ ] : \mathbb{X} \times \mathbb{X} \longrightarrow \mathbb{K}(=\mathbb{R},~\mathbb{C}) $ is a semi-inner-product (s.i.p.) if for any $ \alpha,~\beta \in \mathbb{K} $ and for any $ x,~y,~z \in \mathbb{X}, $ it satisfies the following:\\
$ (a) $ $ [\alpha x + \beta y, z] = \alpha [x,z] + \beta [y,z], $\\
$ (b) $ $ [x,x] > 0, $ whenever $ x \neq 0, $\\
$ (c) $ $ |[x,y]|^{2} \leq [x,x] [y,y], $\\
$ (d) $ $ [x,\alpha y] = \overline{\alpha} [x,y]. $
\end{definition} 

Semi-inner-products were introduced by Lumer \cite{L} in order to effectively apply Hilbert space type arguments in the setting of normed linear spaces. It follows from \cite{G} that every normed linear space $ (\mathbb{X},\|\|) $ can be represented as an s.i.p. space $ (\mathbb{X},[~,~]) $ such that for all $ x \in \mathbb{X}, $ we have, $ [x,x] = \| x \|^{2}. $ We note that in general, there can be many compatible s.i.p. corresponding to a given norm. Whenever we speak of a s.i.p. $ [~,~] $ in context of a normed linear space $ \mathbb{X} $, we implicitly assume that  $ [~,~] $ is compatible with the norm, i.e., for all $ x \in \mathbb{X}, $ we have, $ [x,x] = \| x \|^{2}. $ Lumer stated in \cite{L} that there exists a unique s.i.p. on a normed linear space $ \mathbb{X} $ if and only if the space is smooth, i.e., there exists a unique supporting hyperplane to $ B_{\mathbb{X}} $ at each point of $ S_{\mathbb{X}}. $\\

In this paper, we apply the above mentioned concepts to characterize the Birkhoff-James orthogonality set of any compact linear operator between a reflexive Banach space and any Banach space, provided the norm attainment set of the operator is of a particularly nice form. As an application of the above result, we deduce the infinite-dimensional Bhatia-$\breve{S}$emrl theorem for compact operators on a Hilbert space and also obtain a complete characterization of Euclidean spaces among all finite-dimensional Banach spaces. We next prove an interesting correlation between the concept of s.i.p. and the notions of $ x^{+} $  and $ x^{-}. $  This enables us to retrieve the norm of a compact (bounded) linear operator, or functional, in terms of its interaction with its Birkhoff-James orthogonal set. As another application of our study, we present some best approximation type results in the setting of Hilbert spaces and Banach spaces. We observe that the results obtained by us in this paper generalize and improve some of the related results obtained by W\'{o}jcik \cite{W}. We would like to end this section with the remark that the present work may also be viewed as an illustration of the applicability of s.i.p. type arguments in the study of geometry of normed linear spaces.

\section{ Main Results.}
We begin with a complete characterization of Birkhoff-James orthogonality of compact linear operators between  a reflexive Banach space and any Banach space, under an additional assumption on the norm attainment set of one of the operators. For this we need the proposition, the proof of which is easy.

\begin{prop}\label{prop-connected}
Let $\mathbb{X}$ be a normed linear space. Let $A$ be a closed subset of $\mathbb{X}$ such that $A=B\cup C,$ where $B$ and $C$ are connected. If $A$ is not connected, then $B$ and $C$ are closed. 
\end{prop}

\begin{theorem}\label{th.conn}
Let $ \mathbb{X} $ be a reflexive Banach space and $ \mathbb{Y} $ be any Banach space. Let $ T,A \in \mathbb{K}(\mathbb{X},\mathbb{Y}). $ Suppose either $(i)$ or $(ii)$ holds.\\
(i) $M_T$ is a connected subset of $S_{\mathbb{X}}.$\\
(ii) $M_T$ is not connected but $ M_T = D \cup (-D), $ where $ D $ is a non-empty connected subset of $ S_{\mathbb{X}}. $ \\
 Then $ T \perp_{B} A $ if and only if there exists $ x \in M_T $ such that $ Tx \perp_{B} Ax. $ 
\end{theorem}
\begin{proof}
	The sufficient part of the theorem is trivial. We only prove the necessary part. Let $T\perp_B A$. \\
	First let $(i)$ be true. Then consider the sets 
	\[W_1=\{x\in M_T:Ax\in (Tx)^+\},\]
	\[W_2=\{x\in M_T:Ax\in (Tx)^-\}.\]
	Then by \cite[Th. 2.1]{SPM}, $W_1\neq \emptyset,~W_2\neq\emptyset.$ It is easy to check that $W_1$ and $W_2$ are closed. Applying \cite[Prop. 2.2]{Sa}, we obtain $M_T=W_1\cup W_2.$ Since $M_T$ is connected, $W_1\cap W_2\neq \emptyset.$ Let $x\in W_1\cap W_2.$ Then $Ax\in (Tx)^+$ and $Ax\in (Tx)^-.$ Hence, $Tx\perp_B Ax.$\\
	Now, suppose that $(ii)$ is true. Then by Proposition \ref{prop-connected},  $D$ is closed. Again considering 
	\[W_1=\{x\in D:Ax\in (Tx)^+\},\]
	\[W_2=\{x\in D:Ax\in (Tx)^-\}\]
	and proceeding as $(i)$ we can show that there exists $x\in D$ such that $Tx\perp_B Ax.$ This completes the proof of the theorem.
\end{proof}

\begin{remark}
(1) This improves Theorem $ 3.1 $ of \cite{W} by removing the additional assumptions that $ \mathbb{Y} $ needs to be smooth and strictly convex. \\
(2)  This also improves Theorem $ 3.1 $ of \cite{W} from the point of view of norm attainment set. We observe that if $ M_T $ is connected or $ |M_T|=2, $ then $ M_T = D \cup (-D), $ where $ D $ is a connected subset of $ S_{\mathbb{X}}.$ Note that there are operators for which $M_T= D \cup (-D)$ but neither $ M_T $ is connected nor $ |M_T|=2, $ in which case  we can apply the above Theorem \ref{th.conn} but not Theorem $ 3.1 $ of \cite{W}. As for example,  consider $T\in\mathbb{L}(\ell^2_{\infty})$ defined by $T(a,b)=(0,a)$ for $(a,b)\in \ell^2_{\infty},$ then $M_T=\{\pm(1,b):|b|\leq 1\},$ which is of the form $D \cup (-D), $ where $ D $ is a connected subset of $ S_{\mathbb{X}}.$ 
\end{remark}

The infinite-dimensional Bhatia-$\breve{S}$emrl theorem for compact operators can be obtained as a corollary to Theorem \ref{th.conn}.

\begin{cor}\label{cor.conn}
Let $ \mathbb{H} $ be an infinite-dimensional Hilbert space. Let $ T,A \in \mathbb{K}(\mathbb{H},\mathbb{H}). $ Then $ T \perp_{B} A $ if and only if there exists $ x \in M_T $ such that $ Tx \perp Ax. $ 
\end{cor}

\begin{proof}
Since every Hilbert space is reflexive and $ T $ is compact, it follows that $ M_T \neq \phi.$ Moreover, it follows from of \cite[Th. 2.2]{SP} that either $ M_T $ is connected or $ |M_T|=2. $ Thus, $ M_T =  D \cup (-D), $ where $ D $ is a non-empty connected subset of $ S_{\mathbb{X}}. $  Therefore, the desired result follows directly from Theorem \ref{th.conn}.
\end{proof}

As an application of Corollary \ref{cor.conn}, we now obtain a  characterization of Euclidean spaces among all finite-dimensional Banach spaces. We would like to note that the following theorem improves Theorem $ 4.2 $ of \cite{W}.\\

\begin{theorem}
Let $ \mathbb{X} $ be a finite-dimensional Banach space. Then the following statements are equivalent:\\
\noindent $ (1) $ Given any $ T \in \mathbb{L}(\mathbb{X},\mathbb{X}), $ we have, $ M_T $ is the unit sphere of some subspace of $ \mathbb{X}. $\\
 $ (2) $ Given any $ T \in \mathbb{L}(\mathbb{X},\mathbb{X}), $ we have, $ M_T =  D_T \cup (-D_T), $ where $ D_T $ is a connected subset of $ S_{\mathbb{X}}. $ \\
$ (3) $ $ \mathbb{X} $ is an Euclidean space.
\end{theorem}
\begin{proof}
     $ (1) \Rightarrow (2)$ is obvious.\\
      $ (2)\Rightarrow (3). $
      Let $T,A$ be arbitrary linear operators on $\mathbb{X}$ such that $T\perp_B A.$ Since $M_T= D_T \cup (-D_T),$ where $D_T$ is a connected subset of $S_{\mathbb{X}}$, by Theorem \ref{th.conn}, there exists $x\in M_T$ such that $Tx\perp_B Ax.$ Hence, by \cite[Th. 3.3]{BFS}, $\mathbb{X}$ is an Euclidean space.\\
      $(3)\Rightarrow (1).$
      It follows from \cite[Th. 2.2]{SP}.
\end{proof}

Our next objective is to obtain a connection between the geometric concepts of $ x^+, x^- $  introduced in \cite{S} and the notion of s.i.p. in normed linear spaces.

\begin{theorem}\label{theorem:ortho-sip1}
	Let $ \mathbb{X} $ be a normed linear space and $ x, y \in \mathbb{X}. $ Then the following are true: \\
	(i) $ y \in x^{+} $ if and only if there exists a s.i.p. $ [~,~] $ on $ \mathbb{X} $ such that $ [y, x] \geq 0. $\\
	(ii) $ y \in x^{-} $ if and only if there exists a s.i.p. $ [~,~] $ on $ \mathbb{X} $ such that $ [y, x] \leq 0. $
\end{theorem}

\begin{proof}
	We only prove $ (i) $ and note that $(ii)$ can be proved using similar arguments. Suppose there exists a s.i.p. $ [~,~] $ on $ \mathbb{X} $ such that $ [y, x] \geq 0. $ Now, for any $ \lambda \geq 0, $
	$$ \|x+\lambda y\|\|x\| \geq |[x + \lambda y, x]| = |\|x\|^2 + \lambda [y,x]| = \|x\|^2 + \lambda [y,x] \geq \|x\|^2.$$
	Hence, we have $ \|x+\lambda y\| \geq \|x\| $ for all $ \lambda \geq 0. $ Therefore, $ y \in x^{+}. $\\
	Conversely, let  $ y \in x^{+}. $ Then it follows from \cite[Prop. 2.2]{Sa} that either $x\bot_B y$ or $y\in x^+\setminus x^-$. If $x\bot_B y$, then by \cite[Prop. 5.3]{DK}, there exists a s.i.p $[~,~]$ on $\mathbb{X}$ such that $[y,x]=0$. So assume that $y\in x^+\setminus x^-$. Without loss of generality we may assume that $x\in S_{\mathbb{X}}$. Firstly, suppose that $ x $ and $ y $ are linearly dependent. Then $ y= \alpha x, $ where $ \alpha \geq 0, $ as $ y \in x^{+}. $ By virtue of Hahn-Banach theorem, for each $v\in S_{\mathbb{X}}$, there exists at least one linear functional $ f_v \in S_{\mathbb{X}^{*}} $ such that $ f_v(v)=\|v\|=1.$ We choose exactly one such $f_v$ for each $v\in S_{\mathbb{X}}$. For $\lambda \in \mathbb{R}$, we choose $f_{\lambda v}\in \mathbb{X}^*$ such that $f_{\lambda v}=\lambda f_v$. Now, define a mapping $ [~,~]: \mathbb{X}\times \mathbb{X}\longrightarrow \mathbb{R} $ by  $ [u,v] = f_{v}(u) $ for all $ u,~v \in \mathbb{X}. $ This mapping clearly is a s.i.p. on $\mathbb{X}$. Moreover, $ [y,x] =f_x(y)= f_x(\alpha x)= \alpha \|x\| \geq 0, $ as $ \alpha \geq 0. $\\
	Now, suppose $ x, ~ y $ are linearly independent. Let $ Z = span \{x,y\}. $ Then by \cite[Cor. 2.2]{J}, there exists $ z \in Z \setminus \{0\} $ such that $ x \bot_{B} z. $ Hence, $y=ax+bz$ for some $a,b\in \mathbb{R}$. Now, by \cite[Th. 2.1]{J}, there exists a linear functional $g\in S_{\mathbb{X}^*}$ such that $g(x)=\|x\|=1$ and $g(z)=0$. As before, for each $v(\neq x)\in S_{\mathbb{X}}$, choose exactly one $ f_v \in S_{\mathbb{X}^{*}} $ such that $ f_v(v)=\|v\|=1$ and for $x$, choose $f_x=g$. For $\lambda \in \mathbb{R},$ choose $f_{\lambda v}=\lambda f_v $ for all $v\in S_{\mathbb{X}}$. Then the mapping $ [~,~]: \mathbb{X}\times \mathbb{X}\longrightarrow \mathbb{R} $ defined by  $ [u,v] = f_{v}(u) $ for all $ u,~v \in \mathbb{X}, $ is a s.i.p. on $\mathbb{X}$. Moreover, $[y,x] = f_{x}(y)= f_{x}(ax+bz)=a .$ We claim that $a\geq 0. $ For if $ a < 0, $ then for all $ \lambda \leq 0, $
	\[ \|x + \lambda y \|=\|x + \lambda y \| \|f_{x}\| \geq f_{x}(x+ \lambda y)=f_{x}((1+ \lambda a)x + \lambda bz)= (1+ \lambda a)\|x\| \geq \|x\|.\] Therefore, $ y \in x^{-}, $ a contradiction. Thus, for any $ y \in x^{+} $ there exists a s.i.p. $ [~,~] $ on $ \mathbb{X} $ such that $ [y, x] \geq 0. $ This completes the proof of $ (i)$ and establishes the theorem.  
\end{proof}

As an useful application of the above result, we generalize Lemma $ 5.1 $ of \cite{W} to the setting of Banach spaces.
 
\begin{theorem}\label{th.normop}
Let $ \mathbb{X} $ be  a reflexive Banach space and $ \mathbb{Y} $ be any Banach space. Let $ T,A \in \mathbb{K}(\mathbb{X},\mathbb{Y}) $ be such that $ T \perp_{B} A. $ Let $\mathcal{O}$ denote the collection of all s.i.p. on $ \mathbb{Y}. $ Then
\begin{align*}
\| T \| &= sup\{ [Tx,y] : x \in S_{\mathbb{X}}, y \in S_{\mathbb{Y}}, [~,~] \in \mathcal{O}, [Ax,y] \geq 0 \}\\
        &= sup\{ [Tx,y] : x \in S_{\mathbb{X}}, y \in S_{\mathbb{Y}}, [~,~] \in \mathcal{O}, [Ax,y] \leq 0 \}.
\end{align*}
\end{theorem}
\begin{proof}
	Since $T\perp_B A,$ by \cite[Th. 2.1]{SPM}, there exists $u,v\in M_T$ such that $Au\in (Tu)^+$ and $Av\in (Tv)^-.$ Now, by Theorem \ref{theorem:ortho-sip1}, there exists s.i.p. $[,]_1,~[,]_2\in \mathcal{O}$ such that $[Au,Tu]_1\geq 0$ and $[Av,Tv]_2\leq 0.$ Therefore, 
	\begin{align*}
	\| T \| & \geq sup\{ [Tx,y] : x \in S_{\mathbb{X}}, y \in S_{\mathbb{Y}}, [~,~] \in \mathcal{O}, [Ax,y] \geq 0 \}\\
	& \geq \Big[Tu,\frac{Tu}{\|Tu\|}\Big]_1\\
	&= \|T\|.
	\end{align*}
	Thus, $\|T\|=sup\{ [Tx,y] : x \in S_{\mathbb{X}}, y \in S_{\mathbb{Y}}, [~,~] \in \mathcal{O}, [Ax,y] \geq 0 \}.$\\
	Similarly, using $[Av,Tv]_2\leq 0$ for $v\in M_T,$ it can be shown that
    $\|T\|= sup\{ [Tx,y] : x \in S_{\mathbb{X}}, y \in S_{\mathbb{Y}}, [~,~] \in \mathcal{O}, [Ax,y] \leq 0 \}.$
\end{proof}

\begin{remark}
We would like to note that Lemma $ 5.1 $ of \cite{W} can be deduced from Theorem \ref{th.conn} by following the same line of arguments, as given in the proof of Theorem \ref{th.normop}.
\end{remark}

Our next result generalizes Theorem \ref{th.normop}, by allowing the linear operators to be bounded instead of compact.

\begin{theorem}
Let $ \mathbb{X},\mathbb{Y} $ be any two normed linear spaces. Let $ T,A \in \mathbb{L}(\mathbb{X},\mathbb{Y}) $ be such that $ T \perp_{B} A. $ Let $\mathcal{O}$ denote the collection of all s.i.p. on $ \mathbb{Y}. $ Let $ \epsilon > 0 $ be arbitrary but fixed after choice. Then $ \| T \| = max~\{ l_{1}(\epsilon), l_{2}(\epsilon) \} = max~\{ l_{1}(\epsilon), l_{3}(\epsilon) \},  $ where,\\
$ l_{1}(\epsilon) = sup\{ [Tx,y] : x \in S_{\mathbb{X}}, y \in S_{\mathbb{Y}}, [~,~] \in \mathcal{O}, |[Ax,y]| < \epsilon \}, $ \\
$ l_{2}(\epsilon) = sup\{ [Tx,y] : x \in S_{\mathbb{X}}, y \in S_{\mathbb{Y}}, [~,~] \in \mathcal{O}, Ax \in (y)^{+\epsilon} \}, $ \\
$ l_{3}(\epsilon) = sup\{ [Tx,y] : x \in S_{\mathbb{X}}, y \in S_{\mathbb{Y}}, [~,~] \in \mathcal{O}, Ax \in (y)^{-\epsilon} \}. $ \\ 
\end{theorem}
\begin{proof}
	Since $T\perp_B A,$ by \cite[Th 2.4]{SPM},  we have, either (a) or (b) holds:\\
	(a) There exists a norming sequence $ \{x_{n}\} $ for $T$ such that $ \|Ax_n\| \rightarrow 0. $\\
	(b) There exist two norming sequences  $ \{x_{n}\}, ~\{y_{n}\} $ for $T$ and $ ~ \{\epsilon_{n}\}, ~ \{\delta_{n}\} \subseteq [0,1) $ such that $ \epsilon_{n} \rightarrow 0 ,$ $ \delta_{n} \rightarrow 0, $  $ Ax_n \in (Tx_n)^{+\epsilon_{n}}$ and $ Ay_n \in (Ty_n)^{-\delta_{n}} $ for all $ n \in \mathbb{N}. $\\
	Suppose $(a)$ holds. Then since $Ax_n \to 0,$ there exists $k\in \mathbb{N}$ such that $\|Ax_n\|<\epsilon$ for all $n\geq k.$ Therefore, for any s.i.p. $[~,~]\in \mathcal{O},$ $|[Ax_n,\frac{Tx_n}{\|Tx_n\|}]|\leq \|Ax_n\|< \epsilon$ for all $n\geq k.$ Now,
	\begin{eqnarray*}
	\|T\|&\geq &l_1(\epsilon)=sup\{ [Tx,y] : x \in S_{\mathbb{X}}, y \in S_{\mathbb{Y}}, [~,~] \in \mathcal{O}, |[Ax,y]| < \epsilon \}\\
	&\geq & sup\Big\{\Big[Tx_n,\frac{Tx_n}{\|Tx_n\|}\Big]:[~,~]\in \mathcal{O},n\geq k\Big\}\\
	&=& sup\{\|Tx_n\|:n\geq k\}\\
	&=& \|T\|.
	\end{eqnarray*}
    Thus, $\|T\|=l_1(\epsilon).$ On the other hand, for $i=1,2,$ $l_i(\epsilon)\leq \|T\|\Rightarrow max\{l_1(\epsilon),l_2(\epsilon)\}\leq \|T\|=l_1(\epsilon)\leq  max\{l_1(\epsilon),l_2(\epsilon)\}. $ Hence, $\|T\|= max\{l_1(\epsilon),l_2(\epsilon)\}.$ Similarly, $\|T\|= max\{l_1(\epsilon),l_3(\epsilon)\}.$\\
    Now, suppose that $(b)$ holds. Then since $\epsilon_n \to 0,$ there exists $m\in \mathbb{N}$ such that $\epsilon_n \leq \epsilon$ for all $n\geq m.$ Clearly, $Ax_n\in (Tx_n)^{+\epsilon_n}$ for all $n\in \mathbb{N}$ implies that $Ax_n\in (Tx_n)^{+\epsilon}$ for all $n\geq m.$ Thus, $Ax_n\in (\frac{Tx_n}{\|Tx_n\|})^{+\epsilon}$ for all $n\geq m.$ Now,
    	\begin{eqnarray*}
    	\|T\|&\geq &l_2(\epsilon)=sup\{ [Tx,y] : x \in S_{\mathbb{X}}, y \in S_{\mathbb{Y}}, [~,~] \in \mathcal{O}, Ax\in (y)^{+\epsilon} \}\\
    	&\geq & sup\Big\{\Big[Tx_n,\frac{Tx_n}{\|Tx_n\|}\Big]:[~,~]\in \mathcal{O},n\geq m\Big\}\\
    	&=& sup\{\|Tx_n\|:n\geq m\}\\
    	&=& \|T\|.
    \end{eqnarray*}
    Thus, $\|T\|=l_2(\epsilon).$ Now, proceeding as before, we have, $\|T\|= max\{l_1(\epsilon),l_2(\epsilon)\}.$ Similarly, using $Ay_n\in(Ty_n)^{-\delta_n},~\delta_n\to 0$ and $\|Ty_n\|\to\|T\|,$ we obtain  $\|T\|= max\{l_1(\epsilon),l_3(\epsilon)\}.$ This completes the proof of the theorem.
\end{proof}

If we consider bounded linear functionals instead of bounded linear operators, then we have the following two theorems:\\

\begin{theorem}\label{th-normfsc}
Let $ \mathbb{X} $ be a normed linear space such that $ \mathbb{X}^{*} $ is strictly convex. Let $ f, g \in \mathbb{X}^{*} $ be such that $ f \perp_{B} g. $ Then
\begin{align*}
\| f \| &= sup\{ f(x) : x \in S_{\mathbb{X}}, g(x) \geq 0 \}\\
        &= sup\{ f(x) : x \in S_{\mathbb{X}}, g(x) \leq 0 \}.
\end{align*}
\end{theorem}
\begin{proof}
	Since $f\perp_B g$ and   $ \mathbb{X}^{*} $ is strictly convex, by \cite[Th. 2.6]{SPM}, there exist $\{x_n\}$, $\{y_n\} $ in $ S_\mathbb{X}$ such that \\
	$(i)|f(x_n)|\rightarrow \|f\|$ and $ |f(y_n)| \rightarrow \|f\|$ as $n \rightarrow \infty$ \\
	$(ii)f(x_n).g(x_n) \geq 0$ and  $f(y_n).g(y_n) \leq 0$ for all $n\in \mathbb{N}$.\\
	Now, 
	\begin{eqnarray*}
	\|f\|&\geq &  sup\{ f(x).\alpha : x \in S_{\mathbb{X}}, |\alpha|=1, g(x).\alpha \geq 0 \}\\
	&\geq & sup\Big\{f(x_n).\frac{f(x_n)}{|f(x_n)|}:n\in \mathbb{N}\Big\}\\
	&=& sup\{|f(x_n)|:n\in \mathbb{N}\}\\
	&=& \|f\|.
   	\end{eqnarray*}
   Thus, $\| f \| = sup\{ f(x).\alpha  : x \in S_{\mathbb{X}}, |\alpha |=1, g(x).\alpha  \geq 0 \}=sup\{ f(x) : x \in S_{\mathbb{X}}, g(x) \geq 0 \}.$ Similarly, using the property of the sequence $\{y_n\}$ we can show that  $\|f\|= sup\{ f(x) : x \in S_{\mathbb{X}}, g(x) \leq 0 \}.$ This completes the proof of the theorem.
\end{proof}

\begin{theorem}
Let $ \mathbb{X} $ be a normed linear space. Let $ f, g \in \mathbb{X}^{*} $ be such that $ f \perp_{B} g. $ Let  $ \epsilon > 0 $ be arbitrary but fixed after choice. Then $ \| f \| = max~\{ l(\epsilon), k_1 \} = max~\{ l (\epsilon), k_2) \},  $ where,\\
$ l(\epsilon) = sup\{ f(x) : x \in S_{\mathbb{X}}, |g(x)| < \epsilon \}, $ \\
$ k_1 = sup\{ f(x) : x \in S_{\mathbb{X}}, g(x) \geq 0 \}, $ \\
$ k_2 = sup\{ f(x) : x \in S_{\mathbb{X}}, g(x) \leq 0 \}. $
\end{theorem}
\begin{proof}
	Since $f\perp_B g,$ by \cite[Th. 2.7]{SPM}, we have, either (a) or (b) holds:\\
	(a) there exists $\{x_n\}$ in $ S_\mathbb{X}$ such that $\mid f(x_n)\mid \rightarrow \|f\| $ and  $g(x_n) \rightarrow 0.$\\
	(b) there exists $\{x_n\}, \{y_n\}$ in $ S_\mathbb{X}$ such that \\
	$(i)|f(x_n)| \rightarrow \|f\|$ and $ |f(y_n)| \rightarrow \|f\|$ as $n \rightarrow \infty$ \\
	$(ii)f(x_n).g(x_n) \geq 0$ and  $f(y_n).g(y_n) \leq 0$ for all $n\in \mathbb{N}$.\\
	Suppose that $(a)$ holds. Since $g(x_n)\to 0,$ there exists $k\in \mathbb{N}$ such that $|g(x_n)|<\epsilon$ for all $n\geq k.$ Now,
	\begin{eqnarray*}
	  \|f\|&\geq&sup\{ f(x). \alpha : x \in S_{\mathbb{X}}, |\alpha |=1, |g(x).\alpha | < \epsilon \}\\
	  &\geq & sup\Big\{f(x_n).\frac{f(x_n)}{|f(x_n)|}:n\geq k\Big\}\\
	  &=& sup\{|f(x_n)|:n\geq k\}\\
	  &=& \|f\|.
	\end{eqnarray*}
    Thus, $\|f\|=sup\{ f(x).\alpha  : x \in S_{\mathbb{X}}, |\alpha |=1, |g(x).\alpha | < \epsilon \}=sup\{ f(x) : x \in S_{\mathbb{X}}, |g(x)| < \epsilon \}=l_1(\epsilon).$ Again, $k_1\leq \|f\|.$ Thus, $max\{l(\epsilon), k_1\}=l(\epsilon)=\|f\|.$ \\
    Now, suppose that $(b)$ holds. Then using similar arguments as in Theorem \ref{th-normfsc}, we obtain, $\|f\|=k_1=k_2.$ Again, $l(\epsilon)\leq \|f\|.$ Thus, $max\{l(\epsilon),k_1\}=max\{l(\epsilon),k_2\}=\|f\|.$ This completes the proof of the theorem.  
\end{proof}

Our next result generalizes Theorem $ 5.2 $ of \cite{W} to the setting of Banach spaces.

\begin{theorem}\label{th-dist}
Let $ \mathbb{X} $ be a reflexive Banach space and $ \mathbb{Y} $ be any Banach space. Let $ T,A \in \mathbb{K}(\mathbb{X},\mathbb{Y}) $ be such that for each $ \lambda \in \mathbb{R}, $ $ M_{T+\lambda A} =  D_{\lambda}\cup (-D_{\lambda}), $ where $ D_{\lambda} $ is a non-empty connected subset of $ S_{\mathbb{X}}. $ Let $ \mathcal{O} $ denote the collection of all s.i.p. on $ \mathbb{Y}. $ Then
\[ dist(T, span\{A\}) = sup\{ [Tx,y] : x \in S_{\mathbb{X}}, y \in S_{\mathbb{Y}}, [~,~] \in \mathcal{O}, [Ax,y]=0 \}. \]
\end{theorem}
\begin{proof}
     By \cite[Th. 2.3]{J}, there exists $\lambda_0 \in \mathbb{R}$ such that $(T+\lambda_0 A) \perp_B A.$ Thus, $\|T+\lambda A\|\geq \|T+\lambda_0 A\|$ for all $\lambda \in \mathbb{R}.$ Therefore, $dist(T,span{A})=\|T+\lambda_0 A\|.$ Since $M_{T+\lambda_0 A}= D_{\lambda_0}\cup (-D_{\lambda_0}),$ where $D_{\lambda_0}$ is a non-empty connected subset of $S_{\mathbb{X}},$ by Theorem \ref{th.conn}, there exists $u\in M_{T+\lambda_0 A},$ such that $(T+\lambda_0 A)u\perp_B Au.$ Now, by \cite[Prop. 5.3]{DK}, there exists a s.i.p. $[~,~]\in \mathcal{O}$ such that $[Au,(T+\lambda_0 A)u]=0.$ Hence,
     \begin{eqnarray*}
     dist(T,span\{A\})&=& \|T+\lambda_0 A\|\\
     &\geq &  sup\{ [(T+\lambda_0 A)x,y] : x \in S_{\mathbb{X}}, y \in S_{\mathbb{Y}}, [~,~] \in \mathcal{O}, [Ax,y]=0 \}\\
     &\geq & \Big[(T+\lambda_0 A)u,\frac{(T+\lambda_0 A)u}{\|(T+\lambda_0 A)u\|}\Big]\\
     &=& \|(T+\lambda_0 A)u\|\\
     &=& \|T+\lambda_0 A\|.
     \end{eqnarray*}
 Thus, 
 \begin{eqnarray*}
 dist(T,span\{A\})&=&sup\{ [(T+\lambda_0 A)x,y] : x \in S_{\mathbb{X}}, y \in S_{\mathbb{Y}}, [~,~] \in \mathcal{O}, [Ax,y]=0 \}\\
 &=& sup\{ [Tx,y] : x \in S_{\mathbb{X}}, y \in S_{\mathbb{Y}}, [~,~] \in \mathcal{O}, [Ax,y]=0 \}.
 \end{eqnarray*}
 This completes the proof of the theorem.
\end{proof}
\begin{remark}
Let us note that Theorem $ 5.2 $ of \cite{W} can be deduced from Theorem \ref{th-dist}, by observing that in case of a Hilbert space, the usual inner product is the only s.i.p. on the space. Moreover, for $ T, A \in \mathbb{K}(\mathbb{H}_1,\mathbb{H}_2), $ where $ \mathbb{H}_1, \mathbb{H}_2 $ are Hilbert spaces, $ M_{T+\lambda A} =  D_{\lambda}\cup (-D_{\lambda}), $ where $ D_{\lambda} $ is a non-empty connected subset of $ S_{\mathbb{H}_1}. $
\end{remark}

In Theorem $ 5.3 $ of \cite{W}, the author has presented Kolmogorov's type characterization of best approximation of a compact linear operator between  Hilbert spaces, in a finite-dimensional subspace. However, our next example illustrates that Theorem $ 5.3 $ of \cite{W} is incorrect.

\begin{example}
Let $ \mathbb{H}_1 = \mathbb{H}_2 = (\mathbb{R}^{3}, \|\|_{2}).  $ Let us consider $ T, A_1, A_2 \in \mathbb{L}(\mathbb{H}_1,\mathbb{H}_2) $ given by:\\
$ T(1,0,0)=(1,0,0),~T(0,1,0)=(0,0,0),~T(0,0,1)=(0,0,0), $\\
$ A_1(1,0,0)=(0,1,0),~A_1(0,1,0)=(1,0,0),~A_1(0,0,1)=(0,1,0), $\\
$ A_2(1,0,0)=(1,0,0),~A_2(0,1,0)=(0,1,0),~A_2(0,0,1)=(1,0,0). $\\

\noindent Let $ \mathcal{Z} = span\{A_1,A_2\}. $ We claim that the following holds:\\
$ dist(T,\mathcal{Z}) < sup\{ |\langle Tx,y \rangle| : x \in S_{\mathbb{H}_1}, y \in S_{\mathbb{H}_2}, B \in \mathcal{Z}, y \perp Bx\}. $\\

\noindent Indeed, it is easy to check that the following are true:\\
$ (1)~ \| T \| = 1 $ and $ M_T = \{ \pm (1,0,0) \}, $\\
$ (2)~ T \perp_{B} A_1 $ and $ T \not\perp_B A_2, $\\
$ (3)~ T \notin \mathcal{Z}. $\\

\noindent Now, we have, $ (1,0,0) \perp A_1(1,0,0). $ Therefore, choosing $ x=y=(1,0,0), $ we have,\\
R.H.S $ \geq | \langle T(1,0,0), (1,0,0) \rangle | = 1. $\\
On the other hand, since $ T \not\perp_B A_2, $ we have, $ T \not\perp_B \mathcal{Z}. $ In other words, we have, L.H.S $ < \| T \| = 1. $ This completes the proof of our claim. 
\end{example} 

Let us now present a correct formulation of Theorem $ 5.3 $ of \cite{W}, in the setting of Banach spaces.

\begin{theorem}\label{th-distsub}
Let $ \mathbb{X} $ be a reflexive Banach space and $ \mathbb{Y} $ be any Banach space. Let $ \mathcal{Z} $ be a finite-dimensional subspace of $ \mathbb{K}(\mathbb{X},\mathbb{Y}). $ Let $ T \in \mathbb{K}(\mathbb{X},\mathbb{Y}) \setminus \mathcal{Z}. $ Let us further assume that for any $ \lambda \in \mathbb{R} $ and for any $ A \in \mathcal{Z}, $ $ M_{T+\lambda A} =  D_{\lambda,A}\cup (-D_{\lambda,A}) , $ where $ D_{\lambda,A} $ is a non-empty connected subset of $ S_{\mathbb{X}}. $ Then there exists $ A_0 \in \mathcal{Z} $ such that
\[ dist(T,\mathcal{Z})= sup\{ [Tx,y] : x \in S_{\mathbb{X}}, y \in S_{\mathbb{Y}}, [A_0x,y]=0 \}. \]
Moreover, $ A_0 $ is the best approximation to $ T $ in $ \mathcal{Z}. $ 
\end{theorem}
\begin{proof}
	Since $\mathcal{Z}$ is a finite-dimensional subspace of $ \mathbb{K}(\mathbb{X},\mathbb{Y}), $ there exists $A_0\in \mathcal{Z}$ such that $dist(T,\mathcal{Z})=dist (T,span\{A_0\})=\|T-A_0\|.$ Thus, by Theorem \ref{th-dist}, we have, 
	\[dist(T,\mathcal{Z})=sup\{ [Tx,y] : x \in S_{\mathbb{X}}, y \in S_{\mathbb{Y}}, [A_0x,y]=0 \}. \]  
\end{proof}

In the setting of Hilbert spaces, Theorem \ref{th-distsub} assumes a simpler form:

\begin{theorem}
Let $ \mathbb{H}_1,\mathbb{H}_2 $ be Hilbert spaces and $ \mathcal{Z} $ be a finite-dimensional subspace of $ \mathbb{K}(\mathbb{H}_1,\mathbb{H}_2). $ Let $ T \in \mathbb{K}(\mathbb{H}_1,\mathbb{H}_2) \setminus \mathcal{Z}. $ Then there exists $ A_0 \in \mathcal{Z} $ such that
\[ dist(T,\mathcal{Z})= sup\{ [Tx,y] : x \in S_{\mathbb{X}}, y \in S_{\mathbb{Y}}, [A_0x,y]=0 \}. \]
Moreover, $ A_0 $ is the best approximation to $ T $ in $ \mathcal{Z}. $  
\end{theorem}

\begin{proof}
    Following the same method of the proof of \cite[Th. 2.2]{SP}, it can be observed that $M_{T+\lambda A}$ is the unit sphere of some subspace of $\mathbb{H}_1$ for all $\lambda \in \mathbb{R},$ when $T,A\in \mathbb{K}(\mathbb{H}_1,\mathbb{H}_2).$ Therefore, for any $ \lambda \in \mathbb{R} $ and for any $ A \in \mathcal{Z}, $ $ M_{T+\lambda A} =  D_{\lambda,A} \cup (-D_{\lambda,A}), $ where $ D_{\lambda,A} $ is a non-empty connected subset of $ S_{\mathbb{H}_1}. $ Thus, by Theorem \ref{th-distsub}, we obtain the desired result.
\end{proof}

\bibliographystyle{amsplain}

\end{document}